\documentclass[a4paper,12pt,reqno]{amsart}

\usepackage{amsmath,amssymb,amsthm}
\usepackage{latexsym}
\usepackage{color}
\usepackage{graphicx}
\usepackage{mathrsfs}
\usepackage{enumerate}
\usepackage[abbrev]{amsrefs}
\usepackage[T1]{fontenc}
\usepackage{mathtools}
\mathtoolsset{showonlyrefs=true}

\allowdisplaybreaks[4]

\theoremstyle{plain}
\newtheorem{thm}{Theorem}[section]
\newtheorem{lemm}[thm]{Lemma}

\newtheorem{cor}[thm]{Corollary}

\theoremstyle{definition}
\newtheorem{df}[thm]{Definition}
\newtheorem{rem}[thm]{Remark}

\makeatletter

\@addtoreset{equation}{section}
\makeatother

\renewcommand{\div}{\operatorname{div}}
\newcommand{\dB}{\dot{B}}
\newcommand{\supp}{\operatorname{supp}}

\renewcommand{\leq}{\leqslant}
\renewcommand{\geq}{\geqslant}

\newcommand{\n}[1]{{\left\|#1\right\|}}
\newcommand{\lp}[1]{\left[#1\right]}
\newcommand{\Mp}[1]{\left\{#1\right\}}
\renewcommand{\sp}[1]{\left(#1\right)}

\begin{document}
\title[$2$D stationary Navier--Stokes equations around a uniform flow]
{Well-posedness of the two-dimensional stationary Navier--Stokes equations around a uniform flow}
\author[M.~Fujii]{Mikihiro Fujii}
\address[M.~Fujii]{Institute of Mathematics for Industry, Kyushu University, Fukuoka 819--0395, Japan}
\email[M.~Fujii]{fujii.mikihiro.096@m.kyushu-u.ac.jp}
\author[H.~Tsurumi]{Hiroyuki Tsurumi}
\address[H.~Tsurumi]{Graduate School of Technology, Industrial and Social Sciences, Tokushima University, Tokushima 770--8506, Japan}
\email[H.~Tsurumi]{tsurumi.hiroyuki@tokushima-u.ac.jp}
\begin{abstract}
In this paper, we consider the solvability of the two-dimensional stationary Navier--Stokes equations on the whole plane $\mathbb{R}^2$.
In \cite{Fu2023}, it was proved that the stationary Navier--Stokes equations on $\mathbb{R}^2$ is ill-posed for solutions around zero.
In contrast, considering solutions around the non-zero constant flow, 
the perturbed system has a better regularity in the linear part, which enables us to prove the unique existence of solutions in the scaling critical spaces of the Besov type.
\end{abstract}
\maketitle


\section{Introduction}\label{sec:intro}
In this paper, we consider the two-dimensional stationary Navier--Stokes equations on the whole plane $\mathbb{R}^2$:
\begin{align}
\label{NS}
    \begin{cases}
        -\Delta v + (v \cdot \nabla)v + \nabla p = f, \qquad & x \in \mathbb{R}^2,\\
        \div v = 0, \qquad & x \in \mathbb{R}^2.
    \end{cases}
\end{align}
Here $v=(v_1(x),v_2(x)) $ and $p=p(x)$ denote the unknown velocity vector and the unknown pressure of the fluid at the point $x=(x_1,x_2)\in \mathbb{R}^2$, respectively, while $f=(f_1(x), f_2(x))$ is the given external force.

This system has been studied for a long time especially in the exterior domain $\mathbb{R}^2\backslash\Omega$, where $\Omega$ is a bounded domain with smooth boundary. 
In this case, the system \eqref{NS} is considered in $\mathbb{R}^2\backslash\Omega$, and the following boundary conditions are assumed:
\begin{align*}
    \begin{cases}
        v(x)=v_b(x),& x \in \partial\Omega,\\
        v(x) \to v_\infty, \qquad & |x| \to \infty,
    \end{cases}
\end{align*}
where $v_b$ is a given boundary data and $v_\infty$ is a constant vector.
This problem becomes hard if $v_b\equiv 0$ and $v_\infty=0$, since in the two-dimensional space, we cannot treat the nonlinear term $(v\cdot\nabla) v$ as a perturbation of the linear term at the spatial infinity. 
This problem stems from the fact that the stationary Stokes equations in the same situation have no solution, which is called the Stokes' paradox. 
Instead, Finn--Smith \cite{FS1967} solved this system with $f\equiv 0$ when $v_\infty$ is small but non-zero, by analysis on the linearized term $-\Delta u+(v_\infty\cdot\nabla) u+\nabla p$, which is called Ossen operator.
On the other hand, in the case $v_\infty=0$, there are some studies solving this problem under special conditions. 
For example, Amick \cite{Am1984} solved when an exterior domain is invariant under the transformation $(x_1, x_2)\mapsto (-x_1, x_2)$, and this work was later generalized  by Pileckas and Russo \cite{PR2012} for instance.
For the disk $\Omega=\{x\in\mathbb{R}^2; |x|< 1\}$, Hillairet and Wittwer \cite{HW2013} constructed stationary solutions {around the swirl flow $\mu x^\perp/|x|^2$ for $f\equiv 0$ and $v_b=\mu x^\perp$} with a sufficiently large constant $\mu$.  
Later on, Yamazaki \cite{Ya2016} showed the existence of solutions for every $f=\div  F$ with $F\in (L^2(\mathbb{R}^2\backslash\Omega))^4$ which is invariant under some symmetric group action.
For the rotating obstacle $\Omega$, we refer to notable studies by Hishida \cite{Hi1999}, Higaki, Maekawa, and Nakahara \cite{HMN2018}, and Gallagher, Higaki, and Maekawa \cite{GHM2019}. 
There are also several previous results on the whole plane $\mathbb{R}^2$.
For example, Yamazaki \cite{Ya2009} solved \eqref{NS} with $v\in L^{2,\infty}(\mathbb{R}^2)$ when given external forces decay sufficiently and have some anti-symmetric forms.
After that, Guillod \cite{Gu2015} showed the existence of a pair $(f, v)$ satisfying  \eqref{NS}, where $f$ is dependent on $v$ and is around an arbitrarily given small vector field having zero integral and decaying faster than $|x|^{-3}$.
Moreover, Guillod and Wittwer \cite{GW2015} found scaling invariant solutions to \eqref{NS} with respect to a rotation conversion.
Recently, Maekawa and the second author \cite{MT2023} showed the existence of classical solutions to external forces given around $\nabla^\perp \phi$, where $\phi$ is {radial compactly} supported function.

However, there still remains the problem on the {unique} existence of small solutions and those continuous dependence {on} given any small external forces in $\mathbb{R}^2$. 
In this paper, we call it the well-posedness problem around $(f, v)=(0, 0)$. 
Such a problem in $\mathbb{R}^n$ with $n \geq 3$ has been already investigated well by previous studies in the framework of critical spaces with regard to the invariant scaling $(f(x), v(x))\mapsto (\lambda^3f(\lambda x), \lambda v(\lambda x))$ { ($\lambda>0$) for \eqref{NS}}.
It is well-known as the Fujita--Kato principle (see
\cite{FK1964}) that considering the solvability of partial differential equations in the scaling critical spaces is important, and there are many results for the non-stationary Navier--Stokes flow; see
the works of Kato \cite{Ka1984}, Cannone--Planchon \cite{CP1996}, and Koch--Tataru \cites{KT2001}.
For the well-posedness of the stationary problem around zero flow on $\mathbb{R}^n$ with $n \geq 3$ in the scaling critical framework, we refer to Chen \cite{Ch1993} for the Lebesgue space, Kozono--Yamazaki \cite{KY1995} for the Morrey space, and Kaneko--Kozono--Shimizu \cite{KKS2019} for the Besov space;
as for the ill-posedness by the lack of continuity of the solution map $f\mapsto v$ in some large Besov norms, see the second author's previous work \cite{Ts2019}.
In contrast, the well-posedness problem around $(f, v)=(0, 0)$ in $\mathbb{R}^2$ seems extremely difficult, because of a similar reason to the problem on exterior domains stated above.
Indeed, the first author \cite{Fu2023} recently showed the ill-posedness by the lack of continuity of the solution map  
{
$\dot B^{2/p-3}_{p,1}(\mathbb{R}^2) \ni f \mapsto u \in \dot B^{2/p-1}_{p,1}(\mathbb{R}^2)$
}
with $1\leq p\leq 2$, which is the strongest ill-posedness framework among scaling invariant Besov spaces,
which implies that the solvability of the two dimensional solutions on the whole plane in just scaling critical spaces is almost hopeless unless we assume some spatial symmetric structure or some strong conditions presented in \cite{Ya2009,Gu2015,Ya2016}.
Therefore, it seems more hopeful to consider the well-posedness problem around the special solutions than around the trivial solution $v=0$.
In this sense, for example, the result of \cite{MT2023} can be regarded as the well-posedness around an exact solution to $f=\nabla^\perp\phi$.
However, there seems to be very few previous results in this direction.

In this paper, we consider the unique existence of solutions around a non-zero uniform flow.
Namely, we treat the well-posedness problem in the whole plane $\mathbb{R}^2$ around $(f, v)=(0, v_{\infty})$, where $v_{\infty} \in \mathbb{R}^2 \setminus \{ 0 \}$ is a given constant vector.
For this purpose, we aim to construct a solution $v=v_{\infty}+u$ with some perturbation $u$.
Using the suitable orthogonal transformation, we may assume that $v_{\infty} = \alpha e_1 = (\alpha, 0)$, where $\alpha = |v_{\infty}| >0$, and $e_1:=(1,0)^T$.
In addition, for the sake of simplicity of discussion, we assume here that the external force has the divergence form $f=\div F$ with some $F=\left\{F_{jk}(x)\right\}_{j,k=1,2}$.
Then, the perturbation $u$ should satisfy
\begin{align}\label{eq:u}
    \begin{cases}
        -\Delta u + \alpha \partial_{x_1}u = \mathbb{P}\div (F - u \otimes u), \qquad & x \in \mathbb{R}^2,\\
        \div u = 0, \qquad & x \in \mathbb{R}^2,
    \end{cases}
\end{align}
where $\mathbb{P}=\left\{ \delta_{jk}+\partial_{x_j}\partial_{x_k}(-\Delta)^{-1}\right\}_{j,k=1,2}$ denotes the Helmholtz projection.
Then we show the well-posedness of the system \eqref{eq:u} around $(F, u)=(0, 0)$ in critical spaces with respect to the invariant scaling $(F, u)\mapsto (\lambda^2 F(\lambda\cdot), \lambda u(\lambda \cdot))$. Here, the invariant scaling means that if $(F, u)$ satisfies \eqref{eq:u}, then $(\lambda^2 F(\lambda\cdot), \lambda u(\lambda \cdot))$ satisfies \eqref{eq:u} with $\alpha$ replaced by $\lambda\alpha$.
Although there is a difference between the whole plane and the exterior domain, our result is similar to the result of Finn--Smith \cite{FS1967}. 
In contrast to their result, which only deal with a trivial external force $f\equiv 0$, it is significant in the sense that we can treat non-trivial external forces. Moreover, in the two-dimensional problem with few previous studies, our result presents a new category of the space of external force guaranteeing the existence of solutions.

Focusing on the anisotropy of the system \eqref{eq:u}, we first introduce the anisotropic Besov spaces. 
More precisely, we consider the Littlewood--Paley decomposition $g=\sum_{j\in\mathbb{Z}}\Delta_j g$ only for $x_2$-direction, and we treat $x_1$ like a time variable. 
From this point of view, we apply the Fourier transform only on $x_2$, and solve the equation as an ordinary {differential} equation on $x_1$. 
In addition, in order to utilize the effect of $\alpha e_1$, we define the hybrid anisotropic Besov norms. 
Concretely, we decompose a function $g$ into the high frequency part $\{\Delta_jg\}_{2^j>\alpha}$ and the low frequency part $\{\Delta_jg\}_{2^j\leq\alpha}$, and take the anisotropic Besov norm differently. 
Here we should multiply the norm of high frequency part by a coefficient $\alpha^{-1/p_1}$ ($p_1$ is the integrability index for $x_1$) in order to show the required estimates and to apply the Banach fixed point principle. 
This assumption implies that if $\alpha$ is large, then the required smallness condition for $F$ becomes relaxed. Conversely, as $\alpha\to +0$, such smallness condition becomes severe.

This paper is organized as follows. 
In the next section, we will reformulate the problem so that we clarify our direction to prove the well-posedness.
In the third section, we will define the hybrid anisotropic Besov norms and state our main theorem.
After that, we show some key estimates in the fourth section, and using them, we prove our main theorem in the fifth section.

\section{Reformulation of the problem}\label{sec:reform}
To rewrite the equation \eqref{eq:u},
we first consider the linearized problem 
\begin{align}\label{eq:lin}
    \begin{cases}
        -\Delta u + \alpha \partial_{x_1}u  = \mathbb{P}\div F, \qquad & x \in \mathbb{R}^2,\\
        \div u = 0, \qquad & x \in \mathbb{R}^2.
    \end{cases}
\end{align}
Here, it follows from $\partial_{x_1}^2(-\Delta)^{-1}=-1-\partial_{x_2}^2(-\Delta)^{-1}$ that 
\begin{align}
    &\mathbb{P}\div F
    ={}
    \begin{bmatrix}
        1 + \partial_{x_1}^2(-\Delta)^{-1} & \partial_{x_1}\partial_{x_2}(-\Delta)^{-1} \\
        \partial_{x_1}\partial_{x_2}(-\Delta)^{-1} & 1 + \partial_{x_2}^2(-\Delta)^{-1} 
    \end{bmatrix}
    \begin{bmatrix}
        \partial_{x_1}F_{11} + \partial_{x_2}F_{12} \\
        \partial_{x_1}F_{21} + \partial_{x_2}F_{22}
    \end{bmatrix}\\
    &\quad
    ={}
    \begin{bmatrix}
        -\partial_{x_2}^2(-\Delta)^{-1} & \partial_{x_1}\partial_{x_2}(-\Delta)^{-1} \\
        \partial_{x_1}\partial_{x_2}(-\Delta)^{-1} & 1 + \partial_{x_2}^2(-\Delta)^{-1} 
    \end{bmatrix}
    \begin{bmatrix}
        \partial_{x_1}F_{11} + \partial_{x_2}F_{12} \\
        \partial_{x_1}F_{21} + \partial_{x_2}F_{22}
    \end{bmatrix}\\
    &\quad
    ={}
    \begin{bmatrix}
        -\partial_{x_2}^2(-\Delta)^{-1}
        (\partial_{x_1}F_{11} + \partial_{x_2}F_{12}) 
        +
        \partial_{x_1}\partial_{x_2}(-\Delta)^{-1}
        (\partial_{x_1}F_{21} + \partial_{x_2}F_{22})\\
        \partial_{x_1}\partial_{x_2}(-\Delta)^{-1}
        (\partial_{x_1}F_{11} + \partial_{x_2}F_{12})
        +
        (1 + \partial_{x_2}^2(-\Delta)^{-1})
        (\partial_{x_1}F_{21} + \partial_{x_2}F_{22})
    \end{bmatrix}\\
    &\quad
    ={}
    -
    \begin{bmatrix}
        \partial_{x_2}F_{21}
        +
        (-\Delta)^{-1}\partial_{x_2}^3(F_{12}+F_{21})
        +
        \partial_{x_1}(-\Delta)^{-1}\partial_{x_2}^2(F_{11}-F_{22})
        \\
        \{
        \partial_{x_2}
        +
        (-\Delta)^{-1}\partial_{x_2}^3
        \}
        (F_{11}-F_{22})
        -
        \partial_{x_1}
        \{ F_{21} + (-\Delta)^{-1}\partial_{x_2}^2(F_{12}+F_{21}) \}
    \end{bmatrix}.
\end{align}
Thus, we see that 
\begin{align}
    &\begin{aligned}
    \partial_{x_1}^2 u_1 - \alpha \partial_{x_1}u_1 + \partial_{x_2}^2 u_1
    ={}&
    \partial_{x_2}F_{21}
    +
    (-\Delta)^{-1}\partial_{x_2}^3(F_{12}+F_{21})\\
    &
    +
    \partial_{x_1}(-\Delta)^{-1}\partial_{x_2}^2(F_{11}-F_{22}),
    \end{aligned}\\
    &\begin{aligned}
    \partial_{x_1}^2 u_2 - \alpha \partial_{x_1}u_2 + \partial_{x_2}^2 u_2
    ={}&
    \partial_{x_2}(F_{11}-F_{22})
    +
    (-\Delta)^{-1}\partial_{x_2}^3(F_{11}-F_{22})\\
    &
    -
    \partial_{x_1}
    \{ F_{21} + (-\Delta)^{-1}\partial_{x_2}^2(F_{12}+F_{21}) \}.
    \end{aligned}
\end{align}
In order to obtain the explicit solution formula for the above equations,
we consider the following equations:
\begin{align}
    &
    \partial_{x_1}^2 w^{(0)} - \alpha \partial_{x_1}w^{(0)} + \partial_{x_2}^2 w^{(0)}
    ={}
    g,\label{w^0}\\
    &
    \partial_{x_1}^2 w^{(1)} - \alpha \partial_{x_1}w^{(1)} + \partial_{x_2}^2 w^{(1)}
    ={}
    \partial_{x_1}g,\label{w^1}\\
    &
    \partial_{x_1}^2 \widetilde{w}^{(0)} - \alpha \partial_{x_1}\widetilde{w}^{(0)} + \partial_{x_2}^2 \widetilde{w}^{(0)}
    ={}
    (-\Delta)^{-1}g,\label{tw^0}\\
    &
    \partial_{x_1}^2 \widetilde{w}^{(1)} - \alpha \partial_{x_1}\widetilde{w}^{(1)} + \partial_{x_2}^2 \widetilde{w}^{(1)}
    ={}
    \partial_{x_1}(-\Delta)^{-1}g.\label{tw^1}
\end{align}
Applying the Fourier transform of \eqref{w^0} with respect to $x_2$, we see that 
\begin{align}\label{Fw^0}
    \partial_{x_1}^2 \widehat{w^{(0)}} - \alpha \partial_{x_1}\widehat{w^{(0)}} - \xi_2^2 \widehat{w^{(0)}}
    ={}
    \widehat{g}
\end{align}
and the corresponding eigen frequencies are given by 
\begin{align}
    \lambda_{\pm}(\xi_2)
    =
    \frac{\alpha \pm \sqrt{\alpha^2 + 4\xi_2^2}}{2}.
\end{align}
Hence, solving the ordinary differential equation \eqref{Fw^0}, we have
\begin{align}
    \widehat{w^{(0)}}(x_1,\xi_2)
    ={}&
    -
    \frac{1}{\sqrt{\alpha^2 + 4\xi_2^2}}
    \int_{-\infty}^{x_1}e^{\lambda_{-}(\xi_2)(x_1-y_1)}\widehat{g}(y_1,\xi_2)dy_1\\
    &
    -
    \frac{1}{\sqrt{\alpha^2 + 4\xi_2^2}}
    \int_{x_1}^{\infty}e^{\lambda_{+}(\xi_2)(x_1-y_1)}\widehat{g}(y_1,\xi_2)dy_1\\
    =:{}&
    \widehat{\mathcal{D}^{(0)}}[g](x_1,\xi_2).
\end{align}
For the solution to \eqref{w^1}, by the similar computation and the integration by parts, there holds
\begin{align}
    \widehat{w^{(1)}}(x_1,\xi_2)
    ={}&
    -
    \frac{1}{\sqrt{\alpha^2 + 4\xi_2^2}}
    \int_{-\infty}^{x_1}e^{\lambda_{-}(\xi_2)(x_1-y_1)}\partial_{y_1}\widehat{g}(y_1,\xi_2)dy_1\\
    &
    -
    \frac{1}{\sqrt{\alpha^2 + 4\xi_2^2}}
    \int_{x_1}^{\infty}e^{\lambda_{+}(\xi_2)(x_1-y_1)}\partial_{y_1}\widehat{g}(y_1,\xi_2)dy_1\\
    ={}&
    -
    \frac{\lambda_{-}(\xi_2)}{\sqrt{\alpha^2 + 4\xi_2^2}}
    \int_{-\infty}^{x_1}e^{\lambda_{-}(\xi_2)(x_1-y_1)}\widehat{g}(y_1,\xi_2)dy_1\\
    &
    -
    \frac{\lambda_{+}(\xi_2)}{\sqrt{\alpha^2 + 4\xi_2^2}}
    \int_{x_1}^{\infty}e^{\lambda_{+}(\xi_2)(x_1-y_1)}\widehat{g}(y_1,\xi_2)dy_1\\
    =:{}&
    \widehat{\mathcal{D}^{(1)}}[g](x_1,\xi_2).
\end{align}
In order to consider \eqref{tw^0} and \eqref{tw^1}, we remark that it holds
\begin{align}
    &
    \begin{aligned}
    \mathscr{F}_{x_2}
    \left[(-\Delta)^{-1}g\right](x_1,\xi_2)
    ={}&
    \mathscr{F}_{x_1}^{-1}
    \left[\frac{1}{\xi_1^2+\xi_2^2}\mathscr{F}_{x_1,x_2}[g](\xi_1,\xi_2)\right](x_1)\\
    ={}&
    \frac{1}{2|\xi_2|}
    \int_{\mathbb{R}}
    e^{-|\xi_2||x_1-y_1|}
    \widehat{g}(y_1,\xi_2)dy_1,
    \end{aligned}\\
    &
    \begin{aligned}
    \mathscr{F}_{x_2}
    \left[\partial_{x_1}(-\Delta)^{-1}g\right](x_1,\xi_2)
    ={}&
    \mathscr{F}_{x_1}^{-1}
    \left[\frac{i\xi_1}{\xi_1^2+\xi_2^2}\mathscr{F}_{x_1,x_2}[g](\xi_1,\xi_2)\right](x_1)\\
    ={}&
    -\frac{1}{2}
    \int_{\mathbb{R}}
    \operatorname{sgn}(x_1-y_1)
    e^{-|\xi_2||x_1-y_1|}
    \widehat{g}(y_1,\xi_2)dy_1.
    \end{aligned}
\end{align}
Thus, we have
\begin{align}
    \widehat{\widetilde{w}^{(0)}}(x_1,\xi_2)
    ={}&
    -
    \frac{1}{2|\xi_2|\sqrt{\alpha^2 + 4\xi_2^2}}
    \int_{-\infty}^{x_1}e^{\lambda_{-}(\xi_2)(x_1-y_1)}
    \int_{\mathbb{R}}
    e^{-|\xi_2||y_1-z_1|}
    \widehat{g}(z_1,\xi_2)
    dz_1
    dy_1\\
    &{} 
    -
    \frac{1}{2|\xi_2|\sqrt{\alpha^2 + 4\xi_2^2}}
    \int_{x_1}^{\infty}
    e^{\lambda_{+}(\xi_2)(x_1-y_1)}
    \int_{\mathbb{R}}
    e^{-|\xi_2||y_1-z_1|}
    \widehat{g}(z_1,\xi_2)
    dz_1
    dy_1\\
    =:{}&
    \widehat{\widetilde{\mathcal{D}}^{(0)}}[g](x_1,\xi_2)
\end{align}
and 
\begin{align}
    \widehat{\widetilde{w}^{(1)}}(x_1,\xi_2)
    ={}&
    \frac{1}{2\sqrt{\alpha^2 + 4\xi_2^2}}
    \int_{-\infty}^{x_1}e^{\lambda_{-}(\xi_2)(x_1-y_1)}\\
    &\qquad \qquad \qquad  
    \times
    \int_{\mathbb{R}}
    \operatorname{sgn}(y_1-z_1)
    e^{-|\xi_2||y_1-z_1|}
    \widehat{g}(z_1,\xi_2)
    dz_1
    dy_1
    \\
    &
    +
    \frac{1}{2\sqrt{\alpha^2 + 4\xi_2^2}}
    \int_{x_1}^{\infty}
    e^{\lambda_{+}(\xi_2)(x_1-y_1)}\\
    &\quad \qquad \qquad \qquad
    \times
    \int_{\mathbb{R}}
    \operatorname{sgn}(y_1-z_1)
    e^{-|\xi_2||y_1-z_1|}
    \widehat{g}(z_1,\xi_2)
    dz_1
    dy_1
    \\
    =:{}&
    \widehat{\widetilde{\mathcal{D}}^{(1)}}[g](x_1,\xi_2).
\end{align}
Hence, we obtain 
\begin{align}
    &\begin{aligned}
    \widehat{u_1}(x_1,\xi_2)
    ={}&
    \widehat{{\mathcal{D}}^{(0)}}
    \lp{\partial_{x_2}F_{21}}(x_1,\xi_2)
    +
    \widehat{\widetilde{\mathcal{D}}^{(0)}}
    \lp{\partial_{x_2}^3(F_{12}+F_{21})}(x_1,\xi_2)\\
    &
    +
    \widehat{\widetilde{\mathcal{D}}^{(1)}}
    \lp{\partial_{x_2}^2(F_{11}-F_{22})}(x_1,\xi_2)\\
    =:{}&
    \widehat{\mathcal{D}_1}[F](x_1,\xi_2),
    \end{aligned}\\
    &\begin{aligned}
    \widehat{u_2}(x_1,\xi_2)
    ={}&
    \widehat{{\mathcal{D}}^{(0)}}
    \lp{\partial_{x_2}(F_{11}-F_{22})}(x_1,\xi_2)
    +
    \widehat{\widetilde{\mathcal{D}}^{(0)}}
    \lp{\partial_{x_2}^3(F_{11}-F_{22})}(x_1,\xi_2)\\
    &
    -
    \widehat{{\mathcal{D}}^{(1)}}
    \lp{ F_{21}} (x_1,\xi_2)
    {-}
    \widehat{\widetilde{\mathcal{D}}^{(1)}}
    \lp{\partial_{x_2}^2(F_{12}+F_{21}) }(x_1,\xi_2)\\
    =:{}&
    \widehat{\mathcal{D}_2}[F](x_1,\xi_2).
    \end{aligned}
\end{align}
and let $\mathcal{D}[F] = (\mathcal{D}_1[F],\mathcal{D}_2[F])$, where $\mathcal{D}_m[F]:=\mathscr{F}^{-1}_{x_2}\lp{\widehat{\mathcal{D}_m}[F]}$ for $m=1,2$.

From the above calculations, we define the notion of the mild solution to \eqref{eq:u}.
\begin{df}
For a given external force $F$, we say that a vector field $u=(u_1(x),u_2(x))$ is a mild solution to \eqref{eq:u} if $u$ satisfy 
\begin{align}\label{eq:mild-u}
    u = \mathcal{D}[F - u \otimes u].
\end{align}
\end{df}

In the following of this paper, we consider the equation \eqref{eq:mild-u} instead of \eqref{eq:u}.

\section{Main results}
In order to state the main results precisely, we introduce some function spaces.
Let $\phi_0 \in \mathscr{S}(\mathbb{R}_{\xi_2})$ satisfy 
\begin{align}
    \supp \phi_0 \subset \{ \xi_2 \in \mathbb{R}\ ; \ 2^{-1} \leq |\xi_2| \leq 2 \},\qquad
    0 \leq \phi_0(\xi_2) \leq 1,
\end{align}
and
\begin{align}
    \sum_{j \in \mathbb{Z}} 
    \phi_j(\xi_2)=1
    \qquad {\rm for\ all\ }\xi_2 \in \mathbb{R}\setminus \{0\},
\end{align}
where $\phi_j(\xi_2):=\phi_0(2^{-j}\xi_2)$.
Then, the Littlewood-Paley frequency localized operators $\{ \Delta_j \}_{j\in \mathbb
{Z}}$ are defined by 
\begin{align}
    \Delta_j f := \mathscr{F}^{-1}\lp{\phi_j(\xi_2)\mathscr{F}[f]}.
\end{align}
Now, for $1 \leq p_1,p_2,q \leq \infty$ and $s \in \mathbb{R}$,
we define the anisotropic Besov spaces as 
\begin{align}
    \dB_{p_1,p_2;q}^{s}(\mathbb{R}^2)
    :={}&
    \Mp{
    f: \mathbb{R}_{x_1} \to \mathscr{S}'(\mathbb{R}_{x_2})/\mathscr{P}(\mathbb{R}_{x_2})
    \ ; \ 
    \n{f}_{\dB_{p_1,p_2;q}^{s}}<\infty
    },\\
    \n{f}_{\dB_{p_1,p_2;q}^{s}}
    :={}&
    \n{
    \Mp{
    2^{sj}
    \n{\Delta_jf}_{L^{p_1}_{x_1}L^{p_2}_{x_2}}
    }_{j \in \mathbb{Z}}
    }_{\ell^{q}(\mathbb{Z})},
\end{align}
where 
$\mathscr{P}(\mathbb{R})$ denotes the set of all polynomials on $\mathbb{R}$ 
{and we} have used the abbreviation 
$L^{p_1}_{x_1}L^{p_2}_{x_2}:=L^{p_1}(\mathbb{R}_{x_1};L^{p_2}(\mathbb{R}_{x_2}))$.
For $\alpha > 0$, we define the hybrid anisotropic Besov norms as 
\begin{align}
    \n{f}_{\dB_{p_1,p_2;q}^{s}}^{h;\alpha}
    :={}&
    \n{
    \Mp{
    2^{sj}
    \n{\Delta_jf}_{L^{p_1}_{x_1}L^{p_2}_{x_2}}
    }_{2^j > \alpha}
    }_{\ell^{q}},\\
    \n{f}_{\dB_{p_1,p_2;q}^{s}}^{\ell;\alpha}
    :={}&
    \n{
    \Mp{
    2^{sj}
    \n{\Delta_jf}_{L^{p_1}_{x_1}L^{p_2}_{x_2}}
    }_{2^j \leq \alpha}
    }_{\ell^{q}}.
\end{align}
Using this we define the function spaces 
$D_{p_1,p_2;q}^{\alpha}(\mathbb{R}^2)$ and
$S_{p_1,p_2;q}^{\alpha}(\mathbb{R}^2)$ as the sets of all distributions $F$ and $u$ on $\mathbb{R}^2$ satisfying 
\begin{align}
    \n{F}_{D_{p_1,p_2;q}^{\alpha}}
    :={}&
    \alpha^{-\frac{1}{p_1}}
    \n{F}_{\dB_{p_1,p_2;q}^{\frac{2}{p_1} + \frac{1}{p_2} -2}}^{h;\alpha}
    +
    \n{F}_{\dB_{p_1,p_2;q}^{\frac{1}{p_1} + \frac{1}{p_2} -2}}^{\ell;\alpha}
    <
    \infty,\\
    \n{u}_{S_{p_1,p_2;q}^{\alpha}}
    :={}&
    \alpha^{-\frac{1}{p_1}}
    \n{u}_{\dB_{p_1,p_2;q}^{\frac{2}{p_1} + \frac{1}{p_2} -1}}^{h;\alpha}
    +
    \n{u}_{\dB_{p_1,p_2;q}^{\frac{1}{p_1} + \frac{1}{p_2} -1}}^{\ell;\alpha}
    <
    \infty,
\end{align}
respectively.
{
We say $D_{p_1,p_2;q}^{\alpha}(\mathbb{R}^2)$ and
$S_{p_1,p_2;q}^{\alpha}(\mathbb{R}^2)$ are the scaling critical data and solution spaces for \eqref{eq:u}, respectively.
Here, see Remark \ref{rem} for the meaning of the scaling critical spaces.}

Now, we state our main result of this article.
\begin{thm}\label{thm:1}
Let $p_1$, $p_2$, and $q$ satisfy
\begin{align}
    \max\Mp{\frac{1}{3}, \frac{2}{3}\sp{1-\frac{1}{p_2}}}
    <
    \frac{1}{p_1}
    \leq 
    \frac{1}{2},\qquad
    1 \leq p_2 < 4,\qquad
    1 \leq q \leq \infty.
\end{align}
Then, there exists a positive constant $\delta=\delta(p_1,p_2,q)$ and $\varepsilon=\varepsilon(p_1,p_2,q)$ such that 
for any $\alpha>0$ and $F \in D_{p_1,p_2;q}^{\alpha}(\mathbb{R}^2)$ with 
$\n{F}_{D_{p_1,p_2;q}^{\alpha}} \leq \delta$,
\eqref{eq:u} possesses a unique mild solution $u$ in the class 
\begin{align}
    \Mp{
    u \in S_{p_1,p_2;q}^{\alpha}(\mathbb{R}^2)\ ;\ 
    \n{u}_{S_{p_1,p_2;q}^{\alpha}} \leq \varepsilon
    }.
\end{align}
{
Moreover, the data-to-solution map $D_{p_1,p_2;q}^{\alpha}(\mathbb{R}^2) \ni F \mapsto u \in S_{p_1,p_2;q}^{\alpha}(\mathbb{R}^2)$ is Lipschitz continuous.
}
\end{thm}
{
\begin{rem}\label{rem}
    Let us mention some remarks on Theorem \ref{thm:1}.
    \begin{enumerate}
        \item 
        If $(u,F)$ satisfies \eqref{eq:u} with some $\alpha$, then
        $(u_{\lambda}(x),F_{\lambda}(x)):=(\lambda u(\lambda x), \lambda^2 F(\lambda x))$ solves \eqref{eq:u} with $\alpha$ replaced by $\lambda\alpha$ for all $\lambda >0$.
        For this invariant scaling, there holds
        \begin{align}
            \n{F_{\lambda}}_{D_{p_1,p_2;q}^{\lambda \alpha}} 
            ={}&
            \n{F}_{D_{p_1,p_2;q}^{\alpha}} ,\\
            \n{u_{\lambda}}_{S_{p_1,p_2;q}^{\lambda \alpha}}
            ={}&
            \n{u}_{S_{p_1,p_2;q}^{\alpha}}
        \end{align}
        for all dyadic numbers $\lambda>0$,
        which means that the spaces $D_{p_1,p_2;q}^{\alpha}(\mathbb{R}^2)$ and $S_{p_1,p_2;q}^{\alpha}(\mathbb{R}^2)$ are the critical spaces for scalings $F_{\lambda}$ and $u_{\lambda}$, respectively.
        \item 
        In the function space used in Theorem \ref{thm:1}, the information on differentiability is given only for the $x_2$ direction; the study of solvability in a framework that can measure differentiability in both the $x_1$ and $x_2$ directions is a subject for future work.
    \end{enumerate}
\end{rem}
}
\section{Key estimates}
In this section,
we prepare several estimates that play key roles in the proof of our main results.
\begin{lemm}\label{lemm:fre-loc}
There {exist} positive {constants} $c$ and $C$ such that 
\begin{align}
    &
    \n{\Delta_j (\alpha^2 - \partial_{x_2}^2)^{-\frac{1}{2}} f}_{L^p}
    \leq{}
    \frac{C}{\sqrt{\alpha^2+2^{2j}}}
    \n{\Delta_j f}_{L^p},\\
    &
    \n{\Delta_j \lambda_{\pm}(D_2) f}_{L^p}
    \leq{}
    C
    |\lambda_{\pm}(2^j)|
    \n{\Delta_j f}_{L^p},\\
    &
    \n{\Delta_j e^{\mp\lambda_{\pm}(D_2)T}f}_{L^p}
    \leq
    C
    e^{\mp c \lambda_{\pm}(2^j)T}
    \n{\Delta_j f}_{L^p},\\
    &
    \n{\Delta_j e^{-|\partial_{x_2}|T}f}_{L^p}
    \leq
    C
    e^{-c 2^jT}
    \n{\Delta_j f}_{L^p}
\end{align}
for all $1 \leq p \leq \infty$, $\alpha \in \mathbb{R}$, $T >0$, $j \in \mathbb{Z}$ and $f \in \mathscr{S}'(\mathbb{R}_{x_2})$ with $\Delta_j f \in L^p(\mathbb{R}_{x_2})$.
\end{lemm}
\begin{proof}
{ 
The last estimate is shown in \cite{I2015}. Thus, we prove the others.}
It follows from the Hausdorff--Young inequality and $\Delta_j=\widetilde{\Delta_j}\Delta_j$, where $\widetilde{\Delta_j}:=\Delta_{j-1}+\Delta_j+\Delta_{j+1}$ that 
\begin{align}
    \n{ \Delta_j(\alpha^2 - \partial_{x_2}^2)^{-\frac{1}{2}} f}_{L^p}
    ={}&
    \n{ \widetilde{\Delta_j}(\alpha^2 - \partial_{x_2}^2)^{-\frac{1}{2}} \Delta_jf}_{L^p}\\
    ={}&
    \n{ \mathscr{F}^{-1}\lp{ \widetilde{\phi_j}(\xi_2)(\alpha^2 + \xi_2^2 )^{-\frac{1}{2}} } * \Delta_jf}_{L^p}\\
    \leq{}&
    \n{ \mathscr{F}^{-1}\lp{ \widetilde{\phi_j}(\xi_2)(\alpha^2 + \xi_2^2 )^{-\frac{1}{2}} } }_{L^1}
    \n{ \Delta_jf}_{L^p}{,}
\end{align}
{where we have set $\widetilde{\phi_j}:=\phi_{j-1} + \phi_j + \phi_{j+1}$.}
Here, we see by the change of the variable that 
\begin{align}
    &
    \n{ \mathscr{F}^{-1}\lp{ \widetilde{\phi_j}(\xi_2)(\alpha^2 + \xi_2^2 )^{-\frac{1}{2}} }(x_2) }_{L^1(\mathbb{R}_{x_2})}\\
    &\quad ={}
    \n{ \mathscr{F}^{-1}\lp{ \widetilde{\phi_0}(2^{-j}\xi_2)(\alpha^2 + \xi_2^2 )^{-\frac{1}{2}} }(x_2) }_{L^1(\mathbb{R}_{x_2})}\\
    &\quad ={}
    \n{ 2^j\mathscr{F}^{-1}\lp{ \widetilde{\phi_0}(\xi_2)(\alpha^2 + 2^{2j}\xi_2^2 )^{-\frac{1}{2}} }(2^jx_2) }_{L^1(\mathbb{R}_{x_2})}\\
    &\quad ={}
    \n{ \mathscr{F}^{-1}\lp{ \widetilde{\phi_0}(\xi_2)(\alpha^2 + 2^{2j}\xi_2^2 )^{-\frac{1}{2}} }(x_2) }_{L^1(\mathbb{R}_{x_2})}\\
    &\quad \leq{}
    {\left(\int_{\mathbb{R}}
    \frac{1}{1+x_2^2}dx_2\right)^{\frac{1}{2}}}
    \n{ (1 + x_2^2)^{\frac{1}{2}} \mathscr{F}^{-1}\lp{ \widetilde{ \phi_0}(\xi_2)(\alpha^2 + 2^{2j}\xi_2^2 )^{-\frac{1}{2}} } {(x_2)}}_{L^2(\mathbb{R}_{x_2})}\\
    &\quad
    ={}
    C
    \n{ \widetilde{\phi_0}(\xi_2)(\alpha^2 + 2^{2j}\xi_2^2 )^{-\frac{1}{2}}  }_{H^1(\mathbb{R}_{\xi_2})}\\
    &\quad
    \leq {}
    C
    \n{ \widetilde{\phi_0}(\xi_2)(\alpha^2 + 2^{2j}\xi_2^2 )^{-\frac{1}{2}}  }_{L^2(\mathbb{R}_{\xi_2})}\\
    &\qquad
    +
    C
    \n{ \partial_{\xi_2}\widetilde{\phi_0}(\xi_2)
    (\alpha^2 + 2^{2j}\xi_2^2 )^{-\frac{1}{2}}  }_{L^2(\mathbb{R}_{\xi_2})}\\
    &\qquad
    +
    C
    \n{ \widetilde{\phi_0}(\xi_2)2^{2j}\xi_2(\alpha^2 + 2^{2j}\xi_2^2 )^{-\frac{3}{2}}  }_{L^2(\mathbb{R}_{\xi_2})}\\
    &\quad 
    \leq{}
    C
    (\alpha^2 + 2^{2j} )^{-\frac{1}{2}}
    +
    C2^{2j}(\alpha^2 + 2^{2j} )^{-\frac{3}{2}}\\
    &\quad 
    \leq{}
    C
    (\alpha^2 + 2^{2j} )^{-\frac{1}{2}},
\end{align}
which completes the proof of the first estimate.

On the other hand, since
\begin{align}
    |\lambda_+(\xi_2)|>|\lambda_-(\xi_2)|
    ={}
    \frac{2\xi_2^2}{\alpha+\sqrt{\alpha^2+4\xi_2^2}}
    \geq{}
    \xi_2^2(\alpha^2+4\xi_2^2)^{-\frac{1}{2}}
\end{align}
for every $\xi_2$, it holds for $|\xi_2|\sim 1$ that
\begin{align}
    |\partial_{\xi_2}(\lambda_{\pm}(2^j\xi_2))|
    ={}
    2(2^j)^2\xi_2(\alpha^2+4(2^j)^2\xi_2^2)^{-\frac{1}{2}}
    \leq{}
    C
    |\lambda_{\pm}(2^j)|.
\end{align}
Therefore, we have
\begin{align}
    &
    \n{ \widetilde{\phi_0}(\xi_2)\lambda_{\pm}(2^j\xi_2)  }_{H^1(\mathbb{R}_{\xi_2})}\\
    &\quad
    ={}
    \sum_{k=0,1}\n{ \partial^k_{\xi_2}\widetilde{\phi_0}(\xi_2)\lambda_{\pm}(2^j\xi_2)  }_{L^2(\mathbb{R}_{\xi_2})}
    +
    \n{ \widetilde{\phi_0}(\xi_2)\partial_{\xi_2}(\lambda_{\pm}(2^j\xi_2))}_{L^2(\mathbb{R}_{\xi_2})}\\
    &\quad 
    \leq{}
    C
    |\lambda_{\pm}(2^j)|.
\end{align}
Moreover, since $\mp\lambda_{\pm}(\xi_2)T=-|\lambda_{\pm}(\xi_2)|T$, we have
\begin{align}
    &
    \n{ \widetilde{\phi_0}(\xi_2)e^{\mp\lambda_{\pm}(2^j\xi_2)T}  }_{H^1(\mathbb{R}_{\xi_2})}\\
    &\quad
    ={}
    \sum_{k=0,1}\n{ \partial^k_{\xi_2}\widetilde{\phi_0}(\xi_2)e^{\mp\lambda_{\pm}(2^j\xi_2)T}}_{L^2(\mathbb{R}_{\xi_2})}
    +
    \n{ \widetilde{\phi_0}(\xi_2)\partial_{\xi_2}(e^{\mp\lambda_{\pm}(2^j\xi_2)T})}_{L^2(\mathbb{R}_{\xi_2})}\\
    &\quad 
    \leq{}
    C
    e^{\mp c \lambda_{\pm}(2^j)T}
    +
    C
    |\lambda_{\pm}(2^j)|Te^{\mp c \lambda_{\pm}(2^j)T}\\
    &\quad
    \leq{}
    C
    e^{\mp c \lambda_{\pm}(2^j)T}.
\end{align}
Hence by a similar method to the first estimate, we also obtain the second and third estimates.
\end{proof}
Using the above lemma via the linear solution formula derived in Section \ref{sec:reform}, we obtain the estimate for the linear solution to \eqref{eq:lin}.
\begin{lemm}\label{lemm:lin-est}
There exists an absolute positive constant $C$
such that 
\begin{align}
    \n{\mathcal{D}[F]}_{S_{p_1,p_2;q}^{\alpha}}
    \leq{}&
    C
    \alpha^{-\frac{1}{p_1}}
    \n{F}_{\dB_{p_3,p_2;q}^{\frac{1}{p_1} + \frac{1}{p_2} + \frac{1}{p_3} - 2}}^{h;\alpha}\\
    &
    +
    C
    \alpha^{\frac{1}{p_1}-\frac{1}{p_3}}
    \n{F}_{\dB_{p_3,p_2;q}^{-\frac{1}{p_1}+\frac{1}{p_2} + \frac{2}{p_3} - 2}}^{\ell;\alpha}\\
    &
    +
    C
    \alpha^{-1+\frac{1}{p_3}-\frac{1}{p_1}}
    \n{F}_{\dB_{p_3,p_2;q}^{\frac{1}{p_1}+\frac{1}{p_2} - 1}}^{\ell;\alpha}
\end{align}
for all 
$\alpha>0$, 
$1 \leq p_1,p_2,q \leq \infty$,
$1 \leq p_3 \leq p_1$,
and
$F$ provided that the right hand side is finite.
\end{lemm}
\begin{proof}
    For the estimate of $\mathcal{D}^{(0)}[\partial_{x_2}g]$, we see by Lemma \ref{lemm:fre-loc} that 
    \begin{align}
    \n{\Delta_j \mathcal{D}^{(0)}[\partial_{x_2}g](x_1)}_{L^{p_2}_{x_2}}
    \leq{}&
    \frac{C2^j}{\sqrt{\alpha^2 + 2^{2j}}}
    \int_{-\infty}^{x_1}e^{{c}\lambda_{-}(2^j)(x_1-y_1)}\n{\Delta_j g(y_1,\cdot)}_{L^{p_2}}dy_1\\
    &
    +
    \frac{C2^j}{\sqrt{\alpha^2 + 2^{2j}}}
    \int_{x_1}^{\infty}
    e^{{c}\lambda_+(2^j)(x_1-y_1)}
    \n{\Delta_j g(y_1,\cdot)}_{L^{p_2}}dy_1.
    \end{align}
    It follows from the Hausdorff--Young inequality that  
    \begin{align}
    \n{\Delta_j \mathcal{D}^{(0)}[\partial_{x_2}g]}_{L^{p_1}_{x_1}L^{p_2}_{x_2}}
    \leq{}&
    \frac{C2^j}{\sqrt{\alpha^2 + 2^{2j}}}
    \sp{-\lambda_-(2^j)}^{-1+\frac{1}{p_3}-\frac{1}{p_1}}
    \n{\Delta_j g}_{L^{p_3}_{x_1}L^{p_2}_{x_2}}\\
    &
    +
    \frac{C2^j}{\sqrt{\alpha^2 + 2^{2j}}}
    \sp{\lambda_+(2^j)}^{-1+\frac{1}{p_3}-\frac{1}{p_1}}
    \n{\Delta_j g}_{L^{p_3}_{x_1}L^{p_2}_{x_2}}.
    \end{align}
    Using two estimates
    \begin{align}
        &
        \frac{1}{\sqrt{\alpha^2 + 2^{2j}}}
        \leq
        \begin{cases}
        2^{-j} & (2^j > \alpha),\\
        \alpha^{-1} & (2^j \leq \alpha),
        \end{cases}\\
        &
        \lambda_{+}(2^j)
        \geq 
        -\lambda_{-}(2^j)
        =
        \frac{2\xi_2^2}{\alpha+\sqrt{\alpha^2+4\xi_2^2}}
        \geq
        \begin{cases}
        c2^j & (2^j > \alpha),\\
        c\alpha^{-1}2^{2j} & (2^j \leq \alpha),
        \end{cases}
    \end{align}
    we have
    \begin{align}
    \n{\Delta_j \mathcal{D}^{(0)}[\partial_{x_2}g]}_{L^{p_1}_{x_1}L^{p_2}_{x_2}}
    \leq{}&
    \begin{dcases}
    C
    2^{(-1+\frac{1}{p_3}-\frac{1}{p_1})j}\n{\Delta_j g}_{L^{p_3}_{x_1}L^{p_2}_{x_2}} & 2^j > \alpha,\\
    C
    \alpha^{-\frac{1}{p_3}+\frac{1}{p_1}}2^{(-1+\frac{2}{p_3}-\frac{2}{p_1})j}\n{\Delta_j g}_{L^{p_3}_{x_1}L^{p_2}_{x_2}} & 2^j \leq \alpha.
    \end{dcases}
    \end{align}

    For the estimate of $\mathcal{D}^{(1)}[g]$, we see by Lemma \ref{lemm:fre-loc} that 
    \begin{align}
    \n{\Delta_j \mathcal{D}^{(1)}[g](x_1)}_{L^{p_2}_{x_2}}
    \leq{}&
    \frac{C|\lambda_-(2^j)|}{\sqrt{\alpha^2 + 2^{2j}}}
    \int_{-\infty}^{x_1}e^{{c}\lambda_{-}(2^j)(x_1-y_1)}\n{\Delta_j g(y_1,\cdot)}_{L^{p_2}}dy_1\\
    &
    +
    \frac{C\lambda_+(2^j)}{\sqrt{\alpha^2 + 2^{2j}}}
    \int_{x_1}^{\infty}
    e^{{c}\lambda_+(2^j)(x_1-y_1)}
    \n{\Delta_j g(y_1,\cdot)}_{L^{p_2}}dy_1.
    \end{align}
    It follows from the Hausdorff--Young inequality that  
    \begin{align}
    \n{\Delta_j \mathcal{D}^{(1)}[g]}_{L^{p_1}_{x_1}L^{p_2}_{x_2}}
    \leq{}&
    \frac{C}{\sqrt{\alpha^2 + 2^{2j}}}
    \sp{-\lambda_-(2^j)}^{\frac{1}{p_3}-\frac{1}{p_1}}
    \n{\Delta_j g}_{L^{p_3}_{x_1}L^{p_2}_{x_2}}\\
    &
    +
    \frac{C}{\sqrt{\alpha^2 + 2^{2j}}}
    \sp{\lambda_+(2^j)}^{\frac{1}{p_3}-\frac{1}{p_1}}
    \n{\Delta_j g}_{L^{p_3}_{x_1}L^{p_2}_{x_2}}
    \end{align}
    Using two estimates
    \begin{align}
        &
        \frac{1}{\sqrt{\alpha^2 + 2^{2j}}}
        \leq
        \begin{cases}
        2^{-j} & (2^j > \alpha),\\
        \alpha^{-1} & (2^j \leq \alpha),
        \end{cases}\\
        &
        -\lambda_{-}(2^j)
        \leq 
        \lambda_{+}(2^j)
        =
        \frac{\alpha+\sqrt{\alpha^2+{2^{2j+2}}}}{2}
        \leq
        \begin{cases}
        C2^j & (2^j > \alpha),\\
        C\alpha & (2^j \leq \alpha),
        \end{cases}
    \end{align}
    we have
    \begin{align}
    \n{\Delta_j \mathcal{D}^{(1)}[g]}_{L^{p_1}_{x_1}L^{p_2}_{x_2}}
    \leq{}&
    \begin{dcases}
    C2^{(-1+\frac{1}{p_3}-\frac{1}{p_1})j}\n{\Delta_j g}_{L^{p_3}_{x_1}L^{p_2}_{x_2}} & 2^j > \alpha,\\
    C
    \alpha^{-1+\frac{1}{p_3}-\frac{1}{p_1}}
    \n{\Delta_j g}_{L^{p_3}_{x_1}L^{p_2}_{x_2}} & 2^j \leq \alpha.
    \end{dcases}
    \end{align}

    For the estimate of ${\Delta_j \widetilde{\mathcal{D}}^{(0)}[\partial_{x_2}^3g]}$ and $\widetilde{\mathcal{D}}^{(1)}[\partial_{x_2}^2g]$, we see by Lemma \ref{lemm:fre-loc} that 
    \begin{align}
    &
    \n{\Delta_j \widetilde{\mathcal{D}}^{(0)}[\partial_{x_2}^3g](x_1)}_{L^{p_2}_{x_2}}
    +
    \n{\Delta_j \widetilde{\mathcal{D}}^{(1)}[\partial_{x_2}^2g](x_1)}_{L^{p_2}_{x_2}}\\
    &\quad
    \leq{}
    \frac{C2^{2j}}{\sqrt{\alpha^2 + 2^{2j}}}
    \int_{-\infty}^{x_1}e^{{c}\lambda_{-}(2^j)(x_1-y_1)}
    \int_{\mathbb{R}}
    e^{-c2^j|y_1-z_1|}
    \n{\Delta_j g(z_1,\cdot)}_{L^{p_2}}
    dz_1
    dy_1\\
    &\qquad
    +
    \frac{C2^{2j}}{\sqrt{\alpha^2 + 2^{2j}}}
    \int_{x_1}^{\infty}e^{{c}\lambda_{+}(2^j)(x_1-y_1)}
    \int_{\mathbb{R}}
    e^{-c2^j|y_1-z_1|}
    \n{\Delta_j g(z_1,\cdot)}_{L^{p_2}}
    dz_1
    dy_1.
    \end{align}
    Using the Hausdorff--Young inequality twice, we have 
    \begin{align}
    &
    \n{\Delta_j \widetilde{\mathcal{D}}^{(0)}[\partial_{x_2}^3g]}_{L^{p_1}_{x_1}L^{p_2}_{x_2}}
    +
    \n{\Delta_j \widetilde{\mathcal{D}}^{(1)}[\partial_{x_2}^2g]}_{L^{p_1}_{x_1}L^{p_2}_{x_2}}\\
    &\quad 
    \leq{}
    \frac{C2^{2j}}{\sqrt{\alpha^2 + 2^{2j}}}
    \sp{-\lambda_-(2^j)}^{-1+\frac{1}{p_3}-\frac{1}{p_1}}
    { 2^{-j}}
    \n{\Delta_j g}_{L^{p_3}_{x_1}L^{p_2}_{x_2}}\\
    &
    \qquad
    +
    \frac{C2^{2j}}{\sqrt{\alpha^2 + 2^{2j}}}
    \sp{\lambda_+(2^j)}^{-1+\frac{1}{p_3}-\frac{1}{p_1}}
    { 2^{-j}}
    \n{\Delta_j g}_{L^{p_3}_{x_1}L^{p_2}_{x_2}}\\
    &\quad 
    \leq{}
    \begin{dcases}
    C2^{(-1+\frac{1}{p_3}-\frac{1}{p_1})j}\n{\Delta_j g}_{L^{p_3}_{x_1}L^{p_2}_{x_2}} & 2^j > \alpha,\\
    C
    \alpha^{-1+\frac{1}{p_3}-\frac{1}{p_1}}
    \n{\Delta_j g}_{L^{p_3}_{x_1}L^{p_2}_{x_2}} & 2^j \leq \alpha,
    \end{dcases}
    \end{align}
    which completes the proof.
\end{proof}
\begin{rem}
    Let $2 \leq p_1 \leq \infty$ and $p_3:= p_1/2$ in Lemma \ref{lemm:lin-est}.
    Then, since we see by $2/p_1 - 1 \leq 0$ that
    \begin{align}
    &
    \alpha^{-\frac{1}{p_1}}
    \n{F}_{\dB_{p_3,p_2;q}^{\frac{1}{p_1} + \frac{1}{p_2} + \frac{1}{p_3} - 2}}^{h;\alpha}
    =
    \alpha^{-\frac{1}{p_1}}
    \n{F}_{\dB_{\frac{p_1}{2},p_2;q}^{\frac{3}{p_1} + \frac{1}{p_2} - 2}}^{h;\alpha}
    \\
    &
    \alpha^{\frac{1}{p_1}-\frac{1}{p_3}}
    \n{F}_{\dB_{p_3,p_2;q}^{-\frac{1}{p_1}+\frac{1}{p_2} + \frac{2}{p_3} - 2}}^{\ell;\alpha}
    =
    \alpha^{-\frac{1}{p_1}}
    \n{F}_{\dB_{\frac{p_1}{2},p_2;q}^{\frac{3}{p_1}+\frac{1}{p_2} - 2}}^{\ell;\alpha}
    \\
    &
    \begin{aligned}
    \alpha^{-1+\frac{1}{p_3}-\frac{1}{p_1}}
    \n{F}_{\dB_{p_3,p_2;q}^{\frac{1}{p_1}+\frac{1}{p_2} - 1}}^{\ell;\alpha}
    ={}
    \alpha^{-1+\frac{1}{p_1}}
    \n{F}_{\dB_{\frac{p_1}{2},p_2;q}^{\frac{1}{p_1}+\frac{1}{p_2} - 1}}^{\ell;\alpha}
    \leq{}
    \alpha^{-\frac{1}{p_1}}
    \n{F}_{\dB_{\frac{p_1}{2},p_2;q}^{\frac{3}{p_1}+\frac{1}{p_2} - 2}}^{\ell;\alpha},
    \end{aligned}
    \end{align}
    it holds
    \begin{align}
    \n{\mathcal{D}[F]}_{S_{p_1,p_2;q}^{\alpha}}
    \leq{}&
    C
    \alpha^{-\frac{1}{p_1}}
    \n{F}_{\dB_{\frac{p_1}{2},p_2;q}^{\frac{3}{p_1} + \frac{1}{p_2} - 2}}.
\end{align}
\end{rem}
Next, we provide lemmas for nonlinear estimates.
\begin{lemm}\label{lemm:nonlin-est}
Let $\alpha>0$, and let $p_1,p_2,q$ satisfy
\begin{align}
    \max\Mp{\frac{1}{3}, \frac{2}{3}\sp{1-\frac{1}{p_2}}}
    <
    \frac{1}{p_1}
    \leq 
    \frac{1}{2},\qquad
    1 \leq p_2 < 4,\qquad
    1 \leq q \leq \infty.
\end{align}
Then, there exists a positive constant $C=C(p_1,p_2,q)$ such that
\begin{align}
    \alpha^{-\frac{1}{p_1}}
    \n{fg}_{\dB_{\frac{p_1}{2},p_2;q}^{\frac{3}{p_1} + \frac{1}{p_2}-2}}
    \leq {}&
    C
    \n{f}_{S_{p_1,p_2;q}^{\alpha}}
    \n{g}_{S_{p_1,p_2;q}^{\alpha}}
\end{align}
for all $f,g \in S_{p_1,p_2;q}^{\alpha}(\mathbb{R}^2)$.
\end{lemm}
\begin{proof}
By the Bony paraproduct decomposition, it holds
\begin{align}
    fg = T_fg +R(f,g) + T_gf,
\end{align}
where
\begin{align}
     T_fg := 
     \sum_{k \in \mathbb{Z}}
     \sp{
     \sum_{\ell \leq k-3}
     \Delta_{\ell} f
     }
     \Delta_kg,
     \qquad
     R(f,g)
     :=
     \sum_{|k-\ell|\leq 2}
     \Delta_kf\Delta_{\ell}g.
\end{align}
It follows from 
\begin{align}
    \Delta_j 
    T_fg 
    = 
    \Delta_j
    \sum_{|m|\leq 3}
    \sp{
    \sum_{\ell \leq j+m-3}
    \Delta_{\ell} f
    }
    \Delta_{j+m}g
\end{align}
that
\begin{align}
    \n{\Delta_j T_fg}_{L^{\frac{p_1}{2}}_{x_1}L^{p_2}_{x_2}}
    \leq{}&
    C
    \sum_{|m|\leq 3}
    \sum_{\ell \leq j+m-3}
    \n{\Delta_{\ell} f}_{L^{p_1}_{x_1}L^{\infty}_{x_2}}
    \n{\Delta_{j+m}g}_{L^{p_1}_{x_1}L^{p_2}_{x_2}}\\
    \leq{}&
    C
    2^{(1-\frac{1}{p_1})j}
    \n{f}_{\dB_{p_1,\infty;q}^{\frac{1}{p_1}-1}}
    \sum_{|m|\leq 3}
    \n{\Delta_{j+m}g}_{L^{p_1}_{x_1}L^{p_2}_{x_2}}\\
    \leq{}&
    C
    2^{(1-\frac{1}{p_1})j}
    \n{f}_{\dB_{p_1,p_2;q}^{\frac{1}{p_1} + \frac{1}{p_2}-1}}
    \sum_{|m|\leq 3}
    \n{\Delta_{j+m}g}_{L^{p_1}_{x_1}L^{p_2}_{x_2}}.
\end{align}
Multiplying this by $\alpha^{-\frac{1}{p_1}}2^{(\frac{3}{p_1}+\frac{1}{p_2}-2)j}$ and taking $\ell^{q}(\mathbb{Z})$-norm, we have
\begin{align}
    \alpha^{-\frac{1}{p_1}}
    \n{T_fg}_{\dB_{\frac{p_1}{2},p_2;q}^{\frac{3}{p_1} + \frac{1}{p_2}-2}}
    \leq 
    C
    \alpha^{-\frac{1}{p_1}}
    \n{f}_{\dB_{p_1,p_2;q}^{\frac{1}{p_1} + \frac{1}{p_2}-1}}
    \n{g}_{\dB_{p_1,p_2;q}^{\frac{2}{p_1} + \frac{1}{p_2}-1}}.
\end{align}
Using 
\begin{align}
    \n{f}_{\dB_{p_1,p_2;q}^{\frac{1}{p_1} + \frac{1}{p_2}-1}}
    \leq 
    \n{f}_{S_{p_1,p_2;q}^{\alpha}},\qquad
    \alpha^{-\frac{1}{p_1}}
    \n{g}_{\dB_{p_1,p_2;q}^{\frac{2}{p_1} + \frac{1}{p_2}-1}}
    \leq{}
    \n{g}_{S_{p_1,p_2;q}^{\alpha}},
\end{align}
we have 
\begin{align}
    \alpha^{-\frac{1}{p_1}}
    \n{T_fg}_{\dB_{\frac{p_1}{2},p_2;q}^{\frac{3}{p_1} + \frac{1}{p_2}-2}}
    \leq 
    C
    \n{f}_{S_{p_1,p_2;q}^{\alpha}}
    \n{g}_{S_{p_1,p_2;q}^{\alpha}}.
\end{align}
Similarly, it holds
\begin{align}
    \alpha^{-\frac{1}{p_1}}
    \n{T_gf}_{\dB_{\frac{p_1}{2},p_2;q}^{\frac{2}{p_1} + \frac{1}{p_2}-2}}
    \leq 
    C
    \n{f}_{S_{p_1,p_2;q}^{\alpha}}
    \n{g}_{S_{p_1,p_2;q}^{\alpha}}.
\end{align}
For the estimate of $R(f,g)$, we first consider the case $1 \leq p_2 \leq 2$.
By the Bernstein inequality, we see that
\begin{align}
    \alpha^{-\frac{1}{p_1}}
    \n{R(f,g)}_{\dB_{\frac{p_1}{2},p_2;q}^{\frac{3}{p_1} + \frac{1}{p_2}-2}}
    \leq
    C
    \alpha^{-\frac{1}{p_1}}
    \n{R(f,g)}_{\dB_{\frac{p_1}{2},1;q}^{\frac{3}{p_1}-1}}.
\end{align}
Using 
\begin{align}
    \Delta_jR(f,g)
    =
    \Delta_j
    \sum_{{k \geq j-4}}
    \Delta_kf
    \sum_{|k-\ell|\leq 2}
    \Delta_{\ell}g,
\end{align}
we have 
\begin{align}
    &
    2^{(\frac{3}{p_1}-1)j}
    \n{\Delta_jR(f,g)}_{L^{\frac{p_1}{2}}_{x_1}L^{1}_{x_2}}\\
    &\quad 
    \leq{}
    C
    \sum_{{k \geq j-4}}
    2^{(\frac{3}{p_1}-1)(j-k)}
    2^{(\frac{3}{p_1}-1)k}
    \n{\Delta_kf}_{L^{p_1}_{x_1}L^{p_2}_{x_2}}
    \sum_{|k-\ell|\leq 2}
    \n{\Delta_{\ell}g}_{L^{p_1}_{x_1}L^{p_2'}_{x_2}}\\
    &\quad
    \leq{}
    C
    \sum_{{k \geq j-4}}
    2^{(\frac{3}{p_1}-1)(j-k)}
    2^{(\frac{1}{p_1}+\frac{1}{p_2}-1)k}
    \n{\Delta_kf}_{L^{p_1}_{x_1}L^{p_2}_{x_2}}
    \sum_{|k-\ell|\leq 2}
    2^{(\frac{2}{p_1}+\frac{1}{p_2}-1)\ell}
    \n{\Delta_{\ell}g}_{L^{p_1}_{x_1}L^{p_2}_{x_2}}.
\end{align}
It holds by the Hausdorff--Young inequality via $3/p_1-1 >0$ that
\begin{align}
    &
    \alpha^{-\frac{1}{p_1}}
    \n{R(f,g)}_{\dB_{\frac{p_1}{2},1;q}^{\frac{3}{p_1}-1}}\\
    &\quad 
    \leq
    C
    \alpha^{-\frac{1}{p_1}}
    \n{
    \Mp{
    2^{(\frac{1}{p_1}+\frac{1}{p_2}-1)k}
    \n{\Delta_kf}_{L^{p_1}_{x_1}L^{p_2}_{x_2}}
    \sum_{|k-\ell|\leq 2}
    2^{(\frac{2}{p_1}+\frac{1}{p_2}-1)\ell}
    \n{\Delta_{\ell}g}_{L^{p_1}_{x_1}L^{p_2}_{x_2}}
    }_{k \in \mathbb{Z}}
    }_{\ell^q(\mathbb{Z})}\\
    &\quad 
    \leq 
    C
    \alpha^{-\frac{1}{p_1}}
    \n{f}_{\dB_{p_1,p_2;q}^{\frac{1}{p_1} + \frac{1}{p_2}-1}}
    \n{g}_{\dB_{p_1,p_2;q}^{\frac{2}{p_1} + \frac{1}{p_2}-1}}\\
    &\quad 
    \leq{}
    C
    \n{f}_{S_{p_1,p_2;q}^{\alpha}}
    \n{g}_{S_{p_1,p_2;q}^{\alpha}}.
\end{align}
For the case of $p_2 \geq 2$, it holds by {a} similarly argument as above $3/p_1 + 2/p_2 - 2 > 0$ that 
\begin{align}
    \alpha^{-\frac{1}{p_1}}
    \n{R(f,g)}_{\dB_{\frac{p_1}{2},p_2;q}^{\frac{3}{p_1} + \frac{1}{p_2}-2}}
    \leq{}&
    C
    \alpha^{-\frac{1}{p_1}}
    \n{R(f,g)}_{\dB_{\frac{p_1}{2},\frac{p_2}{2};q}^{\frac{3}{p_1}+\frac{2}{p_2}-2}}\\
    \leq{}&
    C
    \alpha^{-\frac{1}{p_1}}
    \n{f}_{\dB_{p_1,p_2;q}^{\frac{1}{p_1} + \frac{1}{p_2}-1}}
    \n{g}_{\dB_{p_1,p_2;q}^{\frac{2}{p_1} + \frac{1}{p_2}-1}}\\
    \leq{}&
    C
    \n{f}_{S_{p_1,p_2;q}^{\alpha}}
    \n{g}_{S_{p_1,p_2;q}^{\alpha}}.
\end{align}
Thus, we complete the proof.
\end{proof}
As a corollary of Lemmas \ref{lemm:lin-est} and \ref{lemm:nonlin-est},
we have the followings.
\begin{cor}\label{cor:lin-nonlin}
Let $\alpha>0$, and let $p_1,p_2,q$ satisfy
\begin{align}
    \max\Mp{\frac{1}{3}, \frac{2}{3}\sp{1-\frac{1}{p_2}}}
    <
    \frac{1}{p_1}
    \leq 
    \frac{1}{2},\qquad
    1 \leq p_2 < 4,\qquad
    1 \leq q \leq \infty.
\end{align}
Then, there exists a positive constant $C_0=C_0(p_1,p_2,q)$ such that
\begin{align}
    \n{\mathcal{D}[F]}_{S_{p_1,p_2;q}^{\alpha}}
    \leq{}&
    C_0
    \n{F}_{D_{p_1,p_2;q}^{\alpha}},\\
    \n{\mathcal{D}[u \otimes v]}_{S_{p_1,p_2;q}^{\alpha}}
    \leq{}&
    C_0
    \n{u}_{S_{p_1,p_2;q}^{\alpha}}
    \n{v}_{S_{p_1,p_2;q}^{\alpha}}
\end{align}
for all $F \in D_{p_1,p_2;q}^{\alpha}(\mathbb{R}^2)$ and $u,v \in S_{p_1,p_2;q}^{\alpha}(\mathbb{R}^2)$.
\end{cor}
\begin{proof}
Using Lemma \ref{lemm:lin-est} with $p_3 = p_1$, we obtain the first estimate.
It follows from Lemma \ref{lemm:lin-est} with $p_3 =p_1/2$ and Lemma \ref{lemm:nonlin-est} that 
\begin{align}
    \n{\mathcal{D}[u \otimes v]}_{S_{p_1,p_2;q}^{\alpha}}
    \leq{}&
    C
    \alpha^{-\frac{1}{p_1}}
    \n{u \otimes v}_{\dB_{\frac{p_1}{2},p_2;q}^{\frac{3}{p_1} + \frac{1}{p_2}-2}}\\
    \leq{}&
    C
    \n{u}_{S_{p_1,p_2;q}^{\alpha}}
    \n{v}_{S_{p_1,p_2;q}^{\alpha}}.
\end{align}
Thus, we complete the proof.
\end{proof}
\section{Proofs of main results}
We are now in a position to present the proof of our main result.
\begin{proof}[Proof of Theorem \ref{thm:1}]
Let $C_0$ be the positive constant appearing in Corollary \ref{cor:lin-nonlin} and let $F \in D_{p_1,p_2;q}^{\alpha}(\mathbb{R}^2)$ satisfy
\begin{align}
    \n{F}_{D_{p_1,p_2;q}^{\alpha}}
    \leq 
    \frac{1}{8C_0^2}.
\end{align}
To construct a mild solution to \eqref{eq:u}, we consider a map 
\begin{align}
    \Phi[u]
    :=
    \mathcal{D}[F - u \otimes u]
\end{align}
on the complete metric space $(X_{p_1,p_2;q}^{\alpha},d_{X_{p_1,p_2;q}^{\alpha}})$ defined by 
\begin{align}
    &
    X_{p_1,p_2;q}^{\alpha}
    :={}
    \Mp{
    u \in S_{p_1,p_2;q}^{\alpha}(\mathbb{R}^2)\ ;\ 
    \n{u}_{S_{p_1,p_2;q}^{\alpha}}
    \leq 
    2C_0\n{F}_{D_{p_1,p_2;q}^{\alpha}}.
    },\\
    &
    d_{X_{p_1,p_2;q}^{\alpha}}(u,v)
    :={}
    \n{ u - v }_{S_{p_1,p_2;q}^{\alpha}}.
\end{align}
Then it follows from Corollary \ref{cor:lin-nonlin} that for any $u,v \in X_{p_1,p_2;q}^{\alpha}$,
\begin{align}
    \n{\Phi[u]}_{S_{p_1,p_2;q}^{\alpha}}
    \leq{}&
    C_0
    \n{F}_{D_{p_1,p_2;q}^{\alpha}}
    +
    C_0
    \n{u}_{S_{p_1,p_2;q}^{\alpha}}^2\\
    \leq{}&
    C_0
    \n{F}_{D_{p_1,p_2;q}^{\alpha}}
    +
    4C_0^3
    \n{F}_{D_{p_1,p_2;q}^{\alpha}}^2\\
    \leq{}&
    2C_0
    \n{F}_{D_{p_1,p_2;q}^{\alpha}}
\end{align}
and 
\begin{align}
    \n{\Phi[u] - \Phi[v]}_{S_{p_1,p_2;q}^{\alpha}}
    \leq{}&
    C_0
    \n{u}_{S_{p_1,p_2;q}^{\alpha}}
    \n{u-v}_{S_{p_1,p_2;q}^{\alpha}}
    +
    C_0
    \n{u-v}_{S_{p_1,p_2;q}^{\alpha}}
    \n{v}_{S_{p_1,p_2;q}^{\alpha}}\\
    \leq{}&
    4C_0^2
    \n{F}_{D_{p_1,p_2;q}^{\alpha}}
    \n{u-v}_{S_{p_1,p_2;q}^{\alpha}}\\
    \leq{}&
    \frac{1}{2}
    \n{u-v}_{S_{p_1,p_2;q}^{\alpha}},
\end{align}
which is implied by $\Phi[u]-\Phi[v] = \mathcal{D}[u \otimes (u-v)] + \mathcal{D}[(u-v)\otimes v]$.
Hence, $\Phi$ is a contraction map on $X_{p_1,p_2;q}^{\alpha}$, 
and thus it follows from the Banach fixed point principle that there exists a unique $u \in X_{p_1,p_2;q}^{\alpha}$ which yields a mild solution to \eqref{eq:u}.

For the uniqueness, if $u$ and $v$ are mild solution to \eqref{eq:u} in the class 
\begin{align}
    \Mp{
    u \in S_{p_1,p_2;q}^{\alpha}(\mathbb{R}^2)\ ;\ 
    \n{u}_{S_{p_1,p_2;q}^{\alpha}}
    \leq 
    \frac{1}{4C_0}
    },
\end{align}
where we note that the above set includes $X_{p_1,p_2;q}^{\alpha}$. 
Then, by Corollary \ref{cor:lin-nonlin}, it holds
\begin{align}
    \n{u-v}_{S_{p_1,p_2;q}^{\alpha}}
    ={}&
    \n{\Phi[u] - \Phi[v]}_{S_{p_1,p_2;q}^{\alpha}}\\
    \leq{}&
    C_0
    \n{u}_{S_{p_1,p_2;q}^{\alpha}}
    \n{u-v}_{S_{p_1,p_2;q}^{\alpha}}
    +
    C_0
    \n{u-v}_{S_{p_1,p_2;q}^{\alpha}}
    \n{v}_{S_{p_1,p_2;q}^{\alpha}}\\
    \leq{}&
    \frac{1}{2}
    \n{u-v}_{S_{p_1,p_2;q}^{\alpha}},
\end{align}
which implies $u=v$.
Thus, we complete the proof.
\end{proof}

\noindent
{\bf Data availability.} \\
Data sharing not applicable to this article as no datasets were generated or analysed during the current study.

\noindent
{\bf Conflict of interest.} \\
The author has declared no conflicts of interest.

\noindent
{\bf Acknowledgements.} \\
The first author was supported by Grant-in-Aid for Research Activity Start-up, Grant Number JP23K19011.
The second author was supported by Grant-in-Aid for JSPS Fellows and Early-Career Scientists, Grant Number JP22KJ1642 and JP24K16946, respectively.

\renewcommand{\refname}{References}

\end{document}